\documentclass[12pt,reqno]{amsart}
\usepackage{amssymb, amsmath,amsthm}
\newtheorem{thm}{Theorem}

\newtheorem{Lemma}{Lemma}
\newtheorem{Corollary}{Corollary}

\numberwithin{defn}{section}
\numberwithin{thm}{section}
\numberwithin{Lemma}{section}
\numberwithin{Corollary}{section}
\numberwithin{Example}{section}
\numberwithin{subsection}{section}
\numberwithin{Remark}{section}
\numberwithin{equation}{section}
\numberwithin{ppn}{section}
\begin{document}
\title
[On the Local convergence  of two-step  Newton  type  ...]
{On the Local convergence  of two-step  Newton type Method in Banach Spaces under generalized Lipschitz Conditions } 
\author{Akanksha Saxena$^1$, J. P. Jaiswal$^2$}
\thanks{$^1$ Department of Mathematics, Maulana Azad National Institute of Technology,  Bhopal, M.P., India-462003, Email: akanksha.sai121@gmail.com}
\thanks{$^2$ Department of Mathematics, Guru Ghasidas Vishwavidyalaya (A Central University),  Bilaspur, C.G., India-495009, Email: asstprofjpmanit@gmail.com.}
\date{}
\maketitle

\textbf{Abstract.} 
The motive of this paper is to discuss the local convergence of a two-step Newton type method of convergence rate three for solving nonlinear equations  in Banach spaces. It is assumed that the first order derivative of nonlinear operator satisfies the generalized Lipschitz i.e. $L$-average condition.
Also, some results on convergence of the same method  in Banach spaces are established under the assumption that the derivative of the operators satisfies the radius or center Lipschitz condition
with a weak $L$-average particularly it is assumed that  $L$  is positive integrable function but not necessarily non-decreasing. 

\section{\bf Introduction}

Consider a nonlinear operator $t:\Omega\subseteq X\to Y$ such that $X$ and $Y$ are two Banach spaces, $\Omega$ is a non-empty open convex subset and $t$ is Fr\'echet differentiable nonlinear operator. Nonlinear problems has so many applications in the field of chemical engineering, transportation, operational research etc. which can be seen in the form of
\begin{eqnarray}\label{eqn:11}
t(x)=0.
\end{eqnarray}
To find the solution of equation $(\ref{eqn:11})$ , Newton's method defined as
\begin{eqnarray}\label{eqn:13}
x_{k+1}&=&x_k-[t'(x_k)]^{-1}t(x_k), \ k\geq0,
\end{eqnarray}
is being preferred though its speed of convergence is low. Newton's method {\cite{Ortega}}, is a well known iterative method which converges quadratically, has been initially studied by Kantorovich \cite{Kantorovich} and then scrutinized by Rall \cite{Rall}. 

Some Newton-type methods with third-order convergence that do not require the computation of second order derivatives have been developed in the refs.(\cite{Homeier}, \cite{Li}, {\cite{Kang}} and {\cite{Traub}}). While the methods of higher $R$-order of convergence are generally not executed frequently despite having fast speed of convergence because its operational cost is high. But the method of higher $R$-order of convergence can be used in the problems of stiff system where fast convergence is required.

In the  numerical point of view, the convergence domain plays a crucial role for the stable behavior of an iterative scheme. Research about the convergence study of Newton methods involves two types: semilocal and local convergence analysis. The semilocal convergence study  is based on the information around an initial point to give criteria ensuring the convergence of iterative methods; meanwhile, the local one is, based on the information around a solution, to find estimates for the radii of the convergence balls. Numerous researchers studied the local convergence analysis for Newton-type, Jarratt-type, Weerakoon-type, etc.  in Banach space setting in the articles ({\cite{Cho}}, \cite{Gonzalez}, \cite{Kanwar} and \cite{Sharma}).

Here, we discuss the local convergence of the classical third-order modification of two-step Newton's method \cite {Argyros} under the $L$-average condition which is expressed as:
\begin{eqnarray}\label{eqn:12}
y_k&=&x_k-[t'(x_k)]^{-1}t(x_k),\nonumber\\
x_{k+1}&=&y_k-[t'(x_k)]^{-1}t(y_k), \ k\geq0.
 \end{eqnarray}
The important characteristic of the method  $(\ref{eqn:12})$  is that: it is simplest and efficient third-order  iterative method,  per iteration it requires two evaluations of the function $t$, one of the first derivative $t'$ and no evaluations of the second derivative $t''$ hence makes it computationally efficient. We find, in the literature, several studies on the weakness and/or extension of the hypotheses made on the underlying operators. 

For re-investigating the local convergence of  Newton's method, generalized Lipschitz conditions was constructed by Wang \cite{Wang},  in which a non-decreasing positive integrable function was used instead of usual Lipschitz constant.  Further, Wang and Li \cite{Wang2}  derived some results on convergence of Newton's method in Banach spaces  when derivative of the operators satisfies the radius or center Lipschitz condition but
with a weak $L$-average. Shakhno {\cite{Shakhno}} have studied the local convergence of the two step Secant-type method when the first-order divided differences satisfy the generalized Lipschitz conditions.


Now, the intriguing question strikes out that whether the radius Lipschitz condition with $L$-average and the non-decreasing of $L$ are necessary for the convergence the third-order modification of Newton’s method. Motivated and inspired by the above mentioned research works in this direction in the present paper, we derived some theorems for  scheme $(\ref{eqn:12})$. In the first result  generalized Lipschitz conditions has been used to study the local convergence which is important to enlarge the convergence region without additional hypotheses along with an error estimate. In the second theorem,  the domain of uniqueness of solution has been derived under center Lipschitz condition. In the last two theorems, weak $L$-average has been used to derive the convergence result of the considered third-order scheme. Also, few corollaries are stated.


The rest part of the paper is structured as follows: section $2$ contains the definitions related to $L$-average conditions. The local convergence and its domain of uniqueness is mentioned in section $3$ and $4$, respectively. Section $5$ deals with the improvement in assumption that the derivative of $t$ satisfies  the radius and center Lipschitz condition with weak $L$-average namely $L$ is assumed to belong to some family of positive integrable functions that are not necessarily non-decreasing for convergence theorems. An  example is also presented to justify the significance of the last result. 



\section{\bf Generalized Lipschitz conditions }
Here, we denote by $V(x^*, r)=\{x:||x-x^*||<r\}$ a ball with radius $r$ and center $x^*$. The condition imposed on the function $t$
\begin{equation}\label{eqn:21}
||t(x)-t(y^\tau)||\leq L(1-\tau)(||x-x^*||+||y-x^*||),\forall \ x,\ y \in V(x^*, r),
\end{equation}
where $y^\tau=x^*+\tau(y-x^*), 0\leq\tau\leq1,$ is usually called radius Lipschitz condition in the ball  $V(x^*, r)$ 
with constant $L$.  Sometimes, if it is only required to satisfy
\begin{eqnarray}\label{eqn:22}
||t(x)-t(x^*)||\leq 2L||x-x^*||, \forall \ x \in V(x^*, r),
\end{eqnarray}
we call it the center Lipschitz condition in the ball $V(x^*, r)$ with constant $L$. Furthermore, $L$ in the Lipschitz conditions does not necessarily have to be
constant but can be a positive integrable function. In this case, conditions $(2.1)-(2.2)$ are respectively, replaced  by
\begin{eqnarray}\label{eqn:23}
||t(x)-t(y^\tau)||\leq \int_{\tau(\rho(x)+\rho(y))}^{\rho(x)+\rho(y)} L(u) du, \forall \ x,\ y \in V(x^*, r), 0\leq\tau \leq1,\nonumber\\
\end{eqnarray}
and
\begin{eqnarray}\label{eqn:24}
||t(x)-t(x^*)||\leq \int_{0}^{2\rho(x)} L(u) du, \forall \ x \in V(x^*, r), 
\end{eqnarray}
where $\rho(x)=||x-x^*||$. At the same time, the corresponding ' Lipschitz conditions'  is referred as to as having the $L$-average or generalized Lipschitz conditions.
Now, we start with the following lemmas, which will be used later in the main theorems.
\begin{Lemma}\label{lm:21}
Suppose that $t$ has a continuous derivative in $V(x^*, r)$ and $[t'(x^*)]^{-1}$ exists.\\
(i) If $[t'(x^*)]^{-1}t'$ satisfies the radius Lipschitz condition with the $L$-average:
{\small
\begin{eqnarray}\label{eqn:25}
||[t'(x^*)]^{-1}(t'(x)-t'(y^\tau))||\leq \int_{\tau(\rho(x)+\rho(y))}^{\rho(x)+\rho(y)} L(u) du, \forall \ x,\ y \in V(x^*, r), 0\leq\tau \leq1, \nonumber\\
\end{eqnarray}
}
where $y^\tau=x^*+\tau(y-x^*)$, $\rho(x)=||x-x^*||$ and $L$ is non-decreasing, then we have
\begin{equation}\label{eqn:26}
\int_{0}^{1}||[t'(x^*)]^{-1}(t'(x)-t'(y^\tau))||\rho(y)d\tau\leq \int_{0}^{\rho(x)+\rho(y)}L(u)\frac{u}{\rho(x)+\rho(y)}\rho(y)du.
\end{equation}
(ii) If $[t'(x^*)]^{-1}t'$ satisfies the center Lipschitz condition with the $L$-average:
\begin{equation}\label{eqn:27}
||[t'(x^*)]^{-1}(t'(x^\tau)-t'(x^*))||\leq \int_{0}^{2\tau\rho(x)} L(u) du, \forall \ x,\ y \in V(x^*, r), 0\leq\tau \leq1,
\end{equation}
where $\rho(x)=||x-x^*||$ and $L$ is non-decreasing, then we have
\begin{equation}\label{eqn:28}
\int_{0}^{1}||[t'(x^*)]^{-1}(t'(x^\tau)-t'(x^*))||\rho(x)d\tau\leq \int_{0}^{2\rho(x)}L(u)\left(\rho(x)-\frac{u}{2}\right)du.
\end{equation}
\end{Lemma}
\begin{proof}
The Lipschitz conditions $(\ref{eqn:25})$ and $(\ref{eqn:27})$, respectively, imply that
{\small
\begin{eqnarray}\label{eqn:29}
\int_{0}^{1}||[t'(x^*)]^{-1}(t'(x)-t'(y^\tau))||\rho(y)d\tau&\leq& \int_{0}^{1}\int_{\tau(\rho(x)+\rho(y))}^{\rho(x)+\rho(y)}L(u) du \rho(y)d\tau\nonumber\\
&=&\int_{0}^{\rho(x)+\rho(y)}L(u)\frac{u}{\rho(x)+\rho(y)}\rho(y)du,\nonumber\\
\int_{0}^{1}||[t'(x^*)]^{-1}(t'(x^\tau)-t'(x^*))||\rho(x)d\tau&\leq& \int_{0}^{1}\int_{0}^{2\tau\rho(x)} L(u) du\rho(x)d\tau\nonumber\\
&=&\int_{0}^{2\rho(x)}L(u)\left(\rho(x)-\frac{u}{2}\right)du.\nonumber
\end{eqnarray}}
where $x^\tau=x^*+\tau(x-x^*)$ and $y^\tau=x^*+\tau(y-x^*)$.
\end{proof}

\begin{Lemma}\cite{Wang2}\label{lm:22}
Suppose that $L$ is positive integrable. Assume that the function $L_a$ defined by relation $(\ref{eqn:56})$ is non-decreasing for some $a$ with $0\le a\le1$.
Then, $for\ each \ b\geq0$, the function $\varphi_{b,a}$ defined by
\begin{equation}\label{eqn:29}
\varphi_{b,a}(f)=\frac{1}{f^{a+b}}\int_{0}^{f}u^bL(u)du,
\end{equation}
is also non-decreasing.
\end{Lemma}

\section{\bf Local convergence of  Newton type  method $(\ref{eqn:12})$}
In this section, we state existence  theorem  under  radius  Lipschitz condition for Newton type  method $(\ref{eqn:12})$. 

\begin{thm}\label{th:31}
Suppose that $t(x^*)=0$, $t$ has a continuous derivative in  $V(x^*, r)$, $[t'(x^*)]^{-1}$ exists and $[t'(x^*)]^{-1}t'$ satisfies the radius Lipschitz condition with the $L$-average:
{\small
\begin{equation}\label{eqn:31}
||[t'(x^*)]^{-1}(t'(x)-t'(y^\tau))||\leq \int_{\tau(\rho(x)+\rho(y))}^{\rho(x)+\rho(y)} L(u) du, \forall \ x,\ y \in V(x^*, r), 0\leq\tau \leq1, 
\end{equation}}
where $y^\tau=x^*+\tau(y-x^*)$, $\rho(x)=||x-x^*||$ and $L$ is non-decreasing. Let $r$ satisfies the relation
\begin{eqnarray}\label{eqn:32}
\frac{ \int_{0}^{2r} L(u)u du}{2r(1-\int_{0}^{2r} L(u) du)}\leq 1.
\end{eqnarray}
Then two step Newton type method $(\ref{eqn:12})$ is convergent for all  \ $x_0\in V(x^*, r)$ and 
\begin{eqnarray}\label{eqn:34}
||y_n-x^*||&\leq&\frac{ \int_{0}^{2\rho(x_n)} L(u)u du}{2(1-\int_{0}^{2\rho(x_n)} L(u) du)}\leq\frac{q_1}{\rho(x_0)}\rho(x_n)^2,
\end{eqnarray}
{\small
\begin{eqnarray}\label{eqn:34a}
||x_{n+1}-x^*||&\leq&\frac{ \int_{0}^{\rho(x_n)+\rho(y_n)} L(u)u du}{(\rho(x_n)+\rho(y_n))(1-\int_{0}^{2\rho(x_n)} L(u) du)}\rho(y_n)\leq\frac{q_2 q_1}{\rho(x_0)\rho(y_0)}\rho(x_n)^3,\nonumber\\
\end{eqnarray}}
where the quantities
{\small
\begin{eqnarray}\label{eqn:36}
q_1=\frac{ \int_{0}^{2\rho(x_0)} L(u)u du}{2\rho(x_0)(1-\int_{0}^{2\rho(x_0)} L(u) du)},\ q_2=\frac{ \int_{0}^{\rho(x_0)+\rho(y_0)} L(u)u du}{(\rho(x_0)+\rho(y_0))(1-\int_{0}^{2\rho(x_0)} L(u) du)},\nonumber\\
\end{eqnarray}}
are less than $1$. Furthermore,
\begin{eqnarray}\label{eqn:33}
||x_n-x^*||\leq C^{3^n-1}||x_0-x^*||,\ n=1,2,\cdots, C=q_1\frac{\rho(x_0)}{\rho(y_0)}.
\end{eqnarray}
\end{thm}
\begin{proof}
On arbitrarily choosing $x_0\in V(x^*, r)$, where $r$ satisfies the relation $(\ref{eqn:32})$, $q_1$ and $q_2$, defined according to inequality $(\ref{eqn:36})$ are less than 1. Indeed, since $L$ is monotone, we get
\begin{eqnarray}
\left( \frac{1}{t_2^2}\int_{0}^{t_2}-\frac{1}{t_1^2}\int_{0}^{t_1}\right)L(u)udu&=\left(\frac{1}{t_2^2}\int_{t_1}^{t_2}+\left(\frac{1}{t_2^2}-\frac{1}{t_1^2}\right)\int_{0}^{t_1}\right)L(u)udu \nonumber \\
&\geq L(t_1)\left( \frac{1}{t_2^2}\int_{t_1}^{t_2}+\left(\frac{1}{t_2^2}-\frac{1}{t_1^2}\right)\int_{0}^{t_1}\right)udu \nonumber \\
&=L(t_1)\left( \frac{1}{t_2^2}\int_{0}^{t_2}-\frac{1}{t_1^2}\int_{0}^{t_1}\right)udu=0, \nonumber
\end{eqnarray}
for $0<t_1<t_2.$ Thus, $\frac{1}{t^2}\int_{0}^{t}L(u)udu$ is non-decreasing with respect to $t$. Now, we have
\begin{eqnarray}
q_1&=&\frac{ \int_{0}^{2\rho(x_0)} L(u)u du}{2\rho(x_0)^2(1-\int_{0}^{2\rho(x_0)} L(u) du)}\rho(x_0)\nonumber\\
&\leq&\frac{ \int_{0}^{2r} L(u)u du}{2r^2(1-\int_{0}^{2r} L(u) du)}\rho(x_0) \leq \frac{||x_0-x^*||}{r}<1\nonumber,\\
q_2&=&\frac{ \int_{0}^{\rho(x_0)+\rho(y_0)} L(u)u du}{(\rho(x_0)+\rho(y_0))^2(1-\int_{0}^{2\rho(x_0)} L(u) du)}(\rho(x_0+\rho(y_0))\nonumber\\
&\leq&\frac{ \int_{0}^{2r} L(u)u du}{2r^2(1-\int_{0}^{2r} L(u) du)}(\rho(x_0)+\rho(y_0)) \leq \frac{||x_0-x^*||+||y_0-x^*||}{2r}<1\nonumber.
\end{eqnarray}
Obviously, if  $x\in V(x^*, r)$, then using center Lipschitz condition with the $L$-average and  the relation $(\ref{eqn:32})$, we have
\begin{eqnarray}\label{eqn:37}
||[t'(x^*)]^{-1}[t'(x)-t'(x^*)]||\leq \int_{0}^{2\rho(x)} L(u) du \leq 1.
\end{eqnarray}
Taking into account the Banach Lemma and the below equation  
\begin{eqnarray}
||I-([t'(x^*)]^{-1}t'(x)-I)||^{-1}=||[t'(x)]^{-1}t'(x^*)||, \nonumber
\end{eqnarray}
we come to following inequality by using the relation $(\ref{eqn:37})$
\begin{eqnarray}\label{eqn:38}
||[t'(x)]^{-1}t'(x^*)||&\leq& \frac{1}{1-\int_{0}^{2\rho(x)} L(u) du}.
\end{eqnarray}
Now, if $x_n\in V(x^*, r)$ then we may write from expression $(\ref{eqn:12})$
\begin{eqnarray}\label{eqn:39}
||y_n-x^*||&=&||x_n-x^*-[t'(x_n)]^{-1}t(x_n)|| \nonumber\\
& =&||[t'(x_n)]^{-1}[t'(x_n)(x_n-x^*)-t(x_n)+t(x^*)]||.
\end{eqnarray}
Expanding $t(x_n)$  along  $x^*$ from Taylor's Expansion, we attain
{\small
\begin{eqnarray}\label{eqn:310}
t(x^*)-t(x_n)+t'(x_n)(x_n-x^*)=t'(x^*)\int_{0}^{1}[t'(x^*)]^{-1}[t'(x_n)-t'(x^\tau)]d\tau(x_n-x^*). \nonumber\\
\end{eqnarray}}
Also, from the expression $(\ref{eqn:31})$ and combining the equations $(\ref{eqn:39})$ and  $(\ref{eqn:310})$, it can written as
{\small
\begin{eqnarray}\label{eqn:311}
||y_n-x^*||&\leq&||[t'(x_n)]^{-1}t'(x^*)||.||\int_{0}^{1}[t'(x^*)]^{-1}[t'(x_n)-t'(x^\tau)]d\tau||.||(x_n-x^*)||\nonumber\\
&\leq& \frac{1}{\int_{0}^{2\rho(x_n)} L(u) du}\int_{0}^{1}\int_{2\tau\rho(x_n)}^{2\rho(x_n)}L(u)du\rho(x_n)d\tau.
\end{eqnarray}}
In view of Lemma $(\ref{lm:21})$ and  the above  inequality, we can obtain the first inequality of expression  $(\ref{eqn:34})$. By similar analogy and using the last sub-step of the scheme  $(\ref{eqn:12})$, we can write
{\small
\begin{eqnarray}\label{eqn:312}
||x_{n+1}-x^*||&\le&||[t'(x_n)]^{-1}t'(x^*)||.||\int_{0}^{1}[t'(x^*)]^{-1}[t'(x_n)-t'(y^\tau)]d\tau||.||(y_n-x^*)||\nonumber\\
&\le&\frac{1}{\int_{0}^{2\rho(x_n)} L(u) du}\int_{0}^{1}\int_{\tau(\rho(x_n)+\rho(y_n))}^{\rho(x_n)+\rho(y_n)}L(u)du\rho(y_n)d\tau.
\end{eqnarray}}
Using Lemma $(\ref{lm:21})$  and above expression, we can get the first inequality of expression $(\ref{eqn:34a})$.
Furthermore, $\rho(x_n)$ and $\rho(y_n)$ are decreasing monotonically, therefore for all $n=0,1,...$, we have
\begin{eqnarray}
||y_n-x^*||&\leq& \frac{ \int_{0}^{2\rho(x_n)} L(u)u du}{2(1-\int_{0}^{2\rho(x_n)} L(u) du)}\nonumber\\
&\leq&\frac{ \int_{0}^{2\rho(x_0)} L(u)u du}{2\rho(x_0)^2(1-\int_{0}^{2\rho(x_n)} L(u) du)}2\rho(x_n)^2\leq \frac{q_1}{\rho(x_0)}\rho(x_n)^2,\nonumber
\end{eqnarray}
Also, by using second inequality of expression $(\ref{eqn:34})$, we have
{\small
\begin{eqnarray}\label{eqn:313}
||x_{n+1}-x^*||&\leq& \frac{ \int_{0}^{\rho(x_n)+\rho(y_n)} L(u)u du}{(\rho(x_n)+\rho(y_n))^2(1-\int_{0}^{2\rho(x_n)} L(u) du)}\rho(y_n).[\rho(x_n)+\rho(y_n)]\nonumber\\
&\leq&\frac{q_2}{\rho(x_0)+\rho(y_0)}[\rho(x_0)\rho(y_n)+\rho(y_n)^2] \leq\frac{q_2 q_1}{\rho(x_0)\rho(y_0)}\rho(x_n)^3, \nonumber\\
\end{eqnarray}}
Hence, we have the complete inequalities of expressions $(\ref{eqn:34})$ and  $(\ref{eqn:34a})$. Also, it can be seen that inequality $(\ref{eqn:33})$ may be easily derived from the expression $(\ref{eqn:313})$.
\end{proof}

\section{\bf The uniqueness ball for the solution of equations}
Here, we derived uniqueness theorem  under  center Lipschitz condition for Newton type  method $(\ref{eqn:12})$. 
\begin{thm}\label{th:41}
Suppose that $t(x^*)=0$, $t$ has a continuous derivative in  $V(x^*, r)$, $[t'(x^*)]^{-1}$ exists and $[t'(x^*)]^{-1}t'$ satisfies the center Lipschitz condition with the $L$-average:
\begin{eqnarray}\label{eqn:41}
|||[t'(x^*)]^{-1}(t'(x)-t'(x^*))||\leq \int_{0}^{2\rho(x)} L(u) du, \forall \ x \in V(x^*, r), 
\end{eqnarray}
where $\rho(x)=||x-x^*||$ and $L$ is positive integrable function. Let $r$ satisfies the relation
\begin{eqnarray}\label{eqn:42}
\frac{ \int_{0}^{2r} L(u)(2r-u) du}{2r}\leq 1.
\end{eqnarray}
Then the equation  $t(x)=0$ has a unique solution $x^*$ in  $V(x^*, r)$.
\end{thm}
\begin{proof}
On arbitrarily choosing $y^*\in V(x^*, r)$, $y^*\neq x^*$ and considering the iteration 
\begin{eqnarray}\label{eqn:43}
||y^*-x^*||&=&||y^*-x^*-[t'(x^*)]^{-1}t(y^*)||. \nonumber\\
&=&||[t'(x^*)]^{-1}[t'(x^*)(y^*-x^*)-t(y^*)+t(x^*)]||.
\end{eqnarray}
Expanding $t(y^*)$  along $\ x^*$ from Taylor's expansion, we have
\begin{eqnarray}\label{eqn:44}
t(x^*)-t(y^*)+t'(x^*)(y^*-x^*)=\int_{0}^{1}[t'(x^*)]^{-1}[t'(y^*)^\tau-t'(x^*)]d\tau(y^*-x^*)\nonumber\\
\end{eqnarray}
Following the expression $(\ref{eqn:41})$ and combining the inequalities $(\ref{eqn:43})$ and $(\ref{eqn:44})$, we can write
{\small
\begin{eqnarray}\label{eqn:45}
||y^*-x^*||&\leq&||[t'(x^*)]^{-1}t'(x^*)||.||\int_{0}^{1}[t'(x^*)]^{-1}[t'(y^*)^\tau-t'(x^*)]d\tau||.||(y^*-x^*)||\nonumber\\
&\leq& \int_{0}^{1}\int_{0}^{2\tau\rho(y^*)}L(u)du\rho(y^*)d\tau.
\end{eqnarray}}
In view of Lemma  $(\ref{lm:21})$ and  expression $(\ref{eqn:45})$, we obtain 
\begin{eqnarray}\label{eqn:46}
||y^*-x^*||&\leq&\frac{1}{2\rho(y^*)}\int_{0}^{2\rho(y^*)}L(u)[2\rho(y^*)-u]du(y^*-x^*)\nonumber\\
&\leq& \frac{ \int_{0}^{2r} L(u)(2r-u) du}{2r}\rho(y^*)\leq||y^*-x^*||.
\end{eqnarray}
but this contradicts our assumption. Thus, we see that $y^*=x^*$. This completes the proof of the theorem.
\end{proof}
 

In particular, assuming that $L$ is a constant, we obtain the following corollaries $(\ref{cr:51})$ and $(\ref{cr:52})$  from theorems $(\ref{th:31})$ and $(\ref{th:41})$, respectively.
								
\begin{Corollary}\label{cr:51}
Suppose that $x^*$ satisfies $t(x^*)=0$, $t$ has a continuous derivative in  $V(x^*, r)$, $[t'(x^*)]^{-1}$ exists and $[t'(x^*)]^{-1}t'$ satisfies the radius Lipschitz condition with the $L$-average:
\begin{eqnarray}\label{eqn:47}
||[t'(x^*)]^{-1}(t'(x)-t'(y^\tau))|| &\leq & L(1-\tau)(||x-x^*||+||y-x^*||), \nonumber\\
&& \forall \ x,\ y \in V(x^*, r), 0\leq\tau\leq 1,
\end{eqnarray}
where $y^\tau=x^*+\tau(y-x^*)$, $\rho(x)=||x-x^*||$ and $L$ is a positive number. Let $r$ satisfies the relation
\begin{eqnarray}\label{eqn:48}
r&=&\frac{1}{3L}.
\end{eqnarray}
Then two-step Newton type method $ (\ref{eqn:12})$ is convergent for all  \ $x_0\in V(x^*, r)$ and 
\begin{eqnarray}\label{eqn:410}
||y_n-x^*||\leq\frac{q_1}{\rho(x_0)}\rho(x_n)^2,
\end{eqnarray}
\begin{eqnarray}\label{eqn:411}
||x_{n+1}-x^*||\leq\frac{q_2 q_1}{\rho(x_0)\rho(y_0)}\rho(x_n)^3,
\end{eqnarray}
where the quantities
\begin{eqnarray}\label{eqn:412}
q_1=\frac{L\rho(x_0)}{[1-2L\rho(x_0)]} ,\ q_2=\frac{L(\rho(x_0)+\rho(y_0))}{2(1-2L\rho(x_0))},
\end{eqnarray}
are less than $1$. Moreover
\begin{eqnarray}\label{eqn:49}
||x_n-x^*||\leq C^{3^n-1}||x_0-x^*||,\ n=1,2,...; C=q_1\frac{\rho(x_0)}{\rho(y_0)}.
\end{eqnarray}
\end{Corollary}
\begin{Corollary}\label{cr:52}
Suppose that $x^*$ satisfies $t(x^*)=0$, $t$ has a continuous derivative in  $V(x^*, r)$, $[t'(x^*)]^{-1}$ exists and $[t'(x^*)]^{-1}t'$ satisfies the center Lipschitz condition with the $L$-average:
\begin{eqnarray}\label{eqn:413}
||[t'(x^*)]^{-1}(t'(x)-t'(x^*))||\leq  2L||x-x^*||, \forall \ x \in V(x^*, r),
\end{eqnarray}
where $\rho(x)=||x-x^*||$ and $L$ is a positive number. Let $r$ fulfills the condition
\begin{eqnarray}\label{eqn:414}
r&=&\frac{1}{L}.
\end{eqnarray}
Then the equation  $t(x)=0$ has a unique solution $x^*$ in  $V(x^*, r)$. Moreover, the ball radius $r$ depends only on $L$.
\end{Corollary}                                               
Now, we will apply our main theorems to some special function $L$ and immediately obtain the following corollaries.
\begin{Corollary}\label{cr:56}
Suppose that $x^*$ satisfies $t(x^*)=0$, $t$ has a continuous derivative in  $V(x^*, r)$, $[t'(x^*)]^{-1}$ exists and $[t'(x^*)]^{-1}t'$ satisfies the radius Lipschitz condition with the $L$-average where given fixed positive constants $\gamma$ and $L>0$ with $L(u)=\gamma +Lu$ i.e.:
\begin{eqnarray}\label{eqn:561}
||[t'(x^*)]^{-1}(t'(x)-t'(y^\tau))|| &\leq & \gamma(1-\tau)(||x-x^*||+||y-x^*||)\nonumber\\
&+&\frac{L}{2}(1-\tau^2)(||x-x^*||+||y-x^*||)^2, \nonumber\\
&& \forall \ x,\ y \in V(x^*, r), 0\leq\tau\leq 1,
\end{eqnarray}
where $y^\tau=x^*+\tau(y-x^*)$, $\rho(x)=||x-x^*||$. Let $r$ satisfies the relation
\begin{eqnarray}\label{eqn:562}
r&=&\frac{-3\gamma+\sqrt{9\gamma^2+40/3L}}{7L}.
\end{eqnarray}
Then two-step Newton type method $ (\ref{eqn:12})$ is convergent for all  \ $x_0\in V(x^*, r)$ and 
\begin{eqnarray}\label{eqn:563}
||y_n-x^*||\leq\frac{q_1}{\rho(x_0)}\rho(x_n)^2,
\end{eqnarray}
\begin{eqnarray}\label{eqn:564}
||x_{n+1}-x^*||\leq\frac{q_2 q_1}{\rho(x_0)\rho(y_0)}\rho(x_n)^3,
\end{eqnarray}
where the quantities
\begin{eqnarray}\label{eqn:565}
q_1&=&\frac{\rho(x_0)[\gamma+4/3\rho(x_0)]}{[1-2\gamma \rho(x_0)-2L\rho(x_0)^2]},\\
q_2&=&\frac{\rho(x_0)+\rho(y_0)[\gamma/2+L/3(\rho(x_0)+\rho(y_0)]}{[1-2\gamma \rho(x_0)-2L\rho(x_0)^2]},\nonumber\\
\end{eqnarray}
are less than $1$. Moreover
\begin{eqnarray}\label{eqn:566}
||x_n-x^*||\leq C^{3^n-1}||x_0-x^*||,\ n=1,2,...; C=q_1\frac{\rho(x_0)}{\rho(y_0)}.
\end{eqnarray}
\end{Corollary}
\begin{Corollary}\label{cr:57}
Suppose that $x^*$ satisfies $t(x^*)=0$, $t$ has a continuous derivative in  $V(x^*, r)$, $[t'(x^*)]^{-1}$ exists and $[t'(x^*)]^{-1}t'$ satisfies the center Lipschitz condition with the $L$-average where given fixed positive constants $\gamma$ and $L>0$ with $L(u)=\gamma +Lu$ i.e.:
\begin{eqnarray}\label{eqn:571}
||[t'(x^*)]^{-1}(t'(x)-t'(x^*))||\leq 2||x-x^*||(\gamma+L||x-x^*||), \forall \ x \in V(x^*, r), \nonumber\\
\end{eqnarray}
where $\rho(x)=||x-x^*||$. Let $r$ satisfies the relation
\begin{eqnarray}\label{eqn:572}
r&=&\frac{2\gamma-\sqrt{4\gamma^2-16/3L}}{8/3L}.
\end{eqnarray}
Then the equation  $t(x)=0$ has a unique solution $x^*$ in  $V(x^*, r)$. Moreover, the ball radius $r$ depends only on $L$ and $\gamma$.
\end{Corollary}

\section{\bf Convergence under weak $L$-average}
This section contains the results on  re-investigation of the conditions and radius of convergence of considered scheme already presented in the first theorem but $L$ is not taken as non-decreasing function. It has been noticed that the convergence order decreases. The second theorem of this section gives a similar result to theorem $(\ref{th:31})$ but under the assumption of center Lipschitz condition.
\begin{thm}\label{th:51}
Suppose that $t(x^*)=0$, $t$ has a continuous derivative in  $V(x^*, r)$, $[t'(x^*)]^{-1}$ exists and $[t'(x^*)]^{-1}t'$ satisfies the radius Lipschitz condition with the $L$-average:
\begin{eqnarray}\label{eqn:51}
||[t'(x^*)]^{-1}(t'(x)-t'(y^\tau))|| \leq  \int_{\tau(\rho(x)+\rho(y))}^{\rho(x)+\rho(y)} L(u) du, \forall \ x,\ y \in V(x^*, r),\nonumber\\
\end{eqnarray}
$0\leq\tau \leq1$, where $y^\tau=x^*+\tau(y-x^*)$, $\rho(x)=||x-x^*||$ and $L$ is positive integrable. Let $r$ satisfies
\begin{eqnarray}\label{eqn:52}
\int_{0}^{2r}L(u)du\leq \frac{1}{2}.
\end{eqnarray}
Then two-step Newton type method $(\ref{eqn:12})$ is convergent for all $x_0\in V(x^*, r)$ and 
\begin{eqnarray}\label{eqn:54}
||y_n-x^*||&\leq&\frac{ \int_{0}^{2\rho(x_n)} L(u)u du}{2(1-\int_{0}^{2\rho(x_n)} L(u) du)}\leq q_1\rho(x_n),
\end{eqnarray}
\begin{eqnarray}\label{eqn:54a}
||x_{n+1}-x^*||&\leq&\frac{ \int_{0}^{\rho(x_n)+\rho(y_n)} L(u)u du}{(\rho(x_n)+\rho(y_n))(1-\int_{0}^{2\rho(x_n)} L(u) du)}\rho(y_n)\leq q_2 q_1\rho(x_n),\nonumber\\
\end{eqnarray}
where the quantities
\begin{eqnarray}\label{eqn:55}
q_1=\frac{ \int_{0}^{2\rho(x_0)} L(u)du}{(1-\int_{0}^{2\rho(x_0)} L(u) du)},\ q_2=\frac{ \int_{0}^{\rho(x_0)+\rho(y_0)} L(u)du}{(1-\int_{0}^{2\rho(x_0)} L(u) du)},
\end{eqnarray}
are less than $1$. Moreover,
\begin{eqnarray}\label{eqn:53}
||x_n-x^*||\leq (q_1q_2)^{n}||x_0-x^*||,\ n=1,2,....
\end{eqnarray}
Furthermore, suppose that the function $L_{a}$ is defined by 
\begin{eqnarray}\label{eqn:56}
L_a(f)=f^{1-a},
\end{eqnarray}
is non-decreasing for some $a$ with $0\le a\le1$ and $r$ satisfies
\begin{equation}\label{eqn:57}
\frac{1}{2r}\int_{0}^{2r}(2r+u)L(u)du\le1.
\end{equation}
Then the two-step Newton type method $(\ref{eqn:12})$ is convergent for all \ $x_0\in V(x^*, r)$ and 
\begin{eqnarray}\label{eqn:58}
||x_n-x^*||\leq C^{(1+2a)^n-1}||x_0-x^*||,\ n=1,2,\cdots, C=Q_1\frac{\rho(x_0)}{\rho(y_0)},
\end{eqnarray}
where the quantity
\begin{eqnarray}\label{eqn:59}
Q_1=\frac{ \int_{0}^{2\rho(x_0)} L(u)u du}{2\rho(x_0)(1-\int_{0}^{2\rho(x_0)} L(u) du)},
\end{eqnarray}
is less than $1$.
\end{thm}
\begin{proof}
On arbitrarily choosing $x_0\in V(x^*, r)$, where $r$ satisfies the relation $(\ref{eqn:52})$, the quantities $q_1$ and $q_2$ defined by the equation  $(\ref{eqn:55})$ are less than 1. Indeed, since $L$ is positive integrable, we can get
\begin{eqnarray}
q_1&=&\frac{ \int_{0}^{2\rho(x_0)} L(u) du}{(1-\int_{0}^{2\rho(x_0)} L(u) du)}
\leq\frac{ \int_{0}^{2r} L(u) du}{(1-\int_{0}^{2r} L(u) du)}
<1\nonumber,\\
q_2&=&\frac{ \int_{0}^{\rho(x_0)+\rho(y_0)} L(u) du}{((1-\int_{0}^{2\rho(x_0)} L(u) du)}
\leq\frac{ \int_{0}^{2r} L(u) du}{(1-\int_{0}^{2r} L(u) du)}
<1\nonumber.
\end{eqnarray}
Obviously, if  $x \in V(x^*, r)$, then using center Lipschitz condition with the $L$ average and the relation $(\ref{eqn:52})$, we have
\begin{eqnarray}\label{eqn:510}
||[t'(x^*)]^{-1}[t'(x)-t'(x^*)]||\leq \int_{0}^{2\rho(x)} L(u) du \forall \ x \in V(x^*, r)
\leq1.
\end{eqnarray}
Taking into account the Banach Lemma and the below equation
\begin{eqnarray}
||I-([t'(x^*)]^{-1}t'(x)-I)||^{-1}=||[t'(x)]^{-1}t'(x^*)||, \nonumber
\end{eqnarray}
we come to following inequality using the relation $(\ref{eqn:510})$
\begin{eqnarray}\label{eqn:511}
||[t'(x)]^{-1}t'(x^*)||&\leq& \frac{1}{1-\int_{0}^{2\rho(x)} L(u) du}.
\end{eqnarray}
Now, if $x_n\in V(x^*, r)$, then we may write from first sub-step of scheme $(\ref{eqn:12})$
\begin{eqnarray}\label{eqn:512}
||y_n-x^*||&=&||x_n-x^*-[t'(x_n)]^{-1}t(x_n)||\nonumber\\
&=&||[t'(x_n)]^{-1}[t'(x_n)(x_n-x^*)-t(x_n)+t(x^*)]||
\end{eqnarray}
Expanding $t(x_n)$  around  $x^*$ from Taylor's Expansion, it can written as
\begin{eqnarray}\label{eqn:513}
t(x^*)-t(x_n)+t'(x_n)(x_n-x^*)=t'(x^*)\int_{0}^{1}[t'(x^*)]^{-1}[t'(x_n)-t'(x^\tau)]d\tau(x_n-x^*).\nonumber\\
\end{eqnarray}
Following the hypothesis  $(\ref{eqn:51})$ and combining the inequalities $(\ref{eqn:512})$ and  $(\ref{eqn:513})$, we may write
\begin{eqnarray}\label{eqn:514}
||y_n-x^*||&\leq&||[t'(x_n)]^{-1}t'(x^*)||.||\int_{0}^{1}[t'(x^*)]^{-1}[t'(x_n)-t'(x^\tau)]d\tau||.||(x_n-x^*)||\nonumber\\
&\leq& \frac{1}{\int_{0}^{2\rho(x_n)} L(u) du}\int_{0}^{1}\int_{2\tau\rho(x_n)}^{2\rho(x_n)}L(u)du\rho(x_n)d\tau.
\end{eqnarray}
Using the results of  Lemma $(\ref{lm:21})$ and the inequality $(\ref{eqn:511})$ in the above expression we can obtain the first inequality of $(\ref{eqn:54})$. By similar analogy for the last sub-step of the scheme $(\ref{eqn:12})$, we can write
{\small
\begin{eqnarray}\label{eqn:515}
||x_{n+1}-x^*||&\le&||[t'(x_n)]^{-1}t'(x^*)||.||\int_{0}^{1}[t'(x^*)]^{-1}[t'(x_n)-t'(y^\tau)]d\tau||.||(y_n-x^*)||\nonumber\\
&\le&\frac{1}{\int_{0}^{2\rho(x_n)} L(u) du}\int_{0}^{1}\int_{\tau(\rho(x_n)+\rho(y_n))}^{\rho(x_n)+\rho(y_n)}L(u)du\rho(y_n)d\tau. 
\end{eqnarray}}
Using Lemma $(\ref{lm:21})$ in the above expression, we can get the first inequality of $ (\ref{eqn:54a})$. Furthermore, $\rho(x_n)$ and $\rho(y_n)$ are decreasing monotonically, therefore for all $n=0,1,...,$ we have
\begin{eqnarray}
||y_n-x^*||\leq \frac{ \int_{0}^{2\rho(x_n)} L(u)u du}{2(1-\int_{0}^{2\rho(x_n)} L(u) du)}
\leq \frac{ \int_{0}^{2\rho(x_0)} L(u) du}{(1-\int_{0}^{2\rho(x_n)} L(u) du)}\rho(x_n)
\leq q_1\rho(x_n).\nonumber
\end{eqnarray}
Using the second inequality of expression $(\ref{eqn:54})$,
 we arrive at
\begin{eqnarray}\label{eqn:517}
||x_{n+1}-x^*||&\leq& \frac{ \int_{0}^{\rho(x_n)+\rho(y_n)} L(u)u du}{(\rho(x_n)+\rho(y_n))(1-\int_{0}^{2\rho(x_n)} L(u) du)}\rho(y_n)\nonumber\\
&\leq&\frac{ \int_{0}^{\rho(x_0)+\rho(y_0)} L(u) du}{(1-\int_{0}^{2\rho(x_0)} L(u) du)}\rho(y_n) \leq q_2q_1\rho(x_n).
\end{eqnarray}
Also, the inequality $(\ref{eqn:53})$ may be easily derived from the expression $(\ref{eqn:517})$. Furthermore, if the function $L_a$ defined by the relation $(\ref{eqn:56})$ is non-decreasing for some $a$ with $0\le a\le1$ and $r$ is determined by inequality $(\ref{eqn:57})$, it follows from the first inequality of expression $(\ref{eqn:54})$ and Lemma $(\ref{lm:22})$ that
\begin{eqnarray*}
||y_n-x^*||&\leq&\frac{ \varphi_{1,a}(2\rho(x_n))2^a}{(1-\int_{0}^{2\rho(x_n)} L(u) du)}\rho(x_n)^{a+1}\\
&\le&\frac{ \varphi_{1,a}(2\rho(x_0))2^a}{(1-\int_{0}^{2\rho(x_n)} L(u) du)}\rho(x_n)^{a+1}=\frac{Q_1}{\rho(x_0)^a}\rho(x_n)^{a+1}.
\end{eqnarray*}
Moreover, from the first inequality of $(\ref{eqn:54a})$ and Lemma $(\ref{lm:22})$, we can write
 \begin{eqnarray*}
||x_{n+1}-x^*||&\leq&\frac{ \varphi_{1,a}(\rho(x_n)+\rho(y_n))(\rho(x_n)+\rho(y_n))^a}{(1-\int_{0}^{2\rho(x_n)} L(u) du)}\rho(y_n),\\
&\leq&\frac{ \varphi_{1,a}(\rho(x_0)+\rho(y_0))(\rho(x_n)+\rho(y_n))^a}{(1-\int_{0}^{2\rho(x_n)} L(u) du)}\rho(y_n),\\
&=&\frac{Q_2Q_1}{\rho(x_0)^a\rho(y_0)^a}\rho(x_n)^{2a+1},
\end{eqnarray*}
where $Q_1<1$ and $Q_2<1$ are determined by the expression $(\ref{eqn:59})$. Also, the inequality $(\ref{eqn:58})$ may be easily derived and hence $x_n$ converges to $x^*$. Thus the proof is completed.
\end{proof}
\begin{thm}\label{th:52}
Suppose that $t(x^*)=0$, $t$ has a continuous derivative in  $V(x^*, r)$, $[t'(x^*)]^{-1}$ exists and $[t'(x^*)]^{-1}t'$ satisfies the center Lipschitz condition with the $L$-average:
\begin{eqnarray}\label{eqn:521}
||[t'(x^*)]^{-1}(t'(x)-t'(x^*))||\leq \int_{0}^{2\rho(x)} L(u) du, \forall \ x \in V(x^*, r), 
\end{eqnarray}
where $\rho(x)=||x-x^*||$ and $L$ is positive integrable function. Let $r$ satisfies
\begin{eqnarray}\label{eqn:522}
\int_{0}^{2r}L(u)du\leq \frac{1}{3}.
\end{eqnarray}
Then the two-step Newton type method $(\ref{eqn:12})$ is convergent for all  \\ \ $x_0\in V(x^*, r)$ and 
\begin{eqnarray}\label{eqn:520a}
||y_n-x^*||&\leq&\frac{2 \int_{0}^{2\rho(x_n)} L(u) du}{1-\int_{0}^{2\rho(x_n)} L(u) du}\rho(x_n)\leq q_1\rho(x_n), \nonumber\\
||x_{n+1}-x^*||&\leq&\frac{ \int_{0}^{2\rho(x_n)} L(u) du+ \int_{0}^{2\rho(y_n)} L(u) du}{1-\int_{0}^{2\rho(x_n)} L(u) du}\rho(y_n)\leq q_2 q_1\rho(x_n),\nonumber\\
\end{eqnarray}
where the quantities
\begin{eqnarray}\label{eqn:524}
q_1=\frac{2 \int_{0}^{2\rho(x_0)} L(u)du}{(1-\int_{0}^{2\rho(x_0)} L(u) du)},\ q_2=\frac{ \int_{0}^{2\rho(x_0)} L(u)du+\int_{0}^{2\rho(y_0)} L(u)du}{(1-\int_{0}^{2\rho(x_0)} L(u) du)}.
\end{eqnarray}
are less than $1$. Moreover,
\begin{eqnarray}\label{eqn:523}
||x_n-x^*||\leq (q_1q_2)^{n}||x_0-x^*||,\ n=1,2,...,
\end{eqnarray}
Furthermore, suppose that the function $L_{a}$ defined by the relation $(\ref{eqn:56})$ is non-decreasing for some $a$ with $0\le a\le1$, then
\begin{eqnarray}\label{eqn:525}
||x_n-x^*||\leq C^{(1+2a)^n-1}||x_0-x^*||,\ n=1,2,\cdots, C=q_1\frac{\rho(x_0)}{\rho(y_0)}.
\end{eqnarray}
and $q_1$ is given by the first expression of the equation $(\ref{eqn:524})$.
\end{thm}
\begin{proof}
Let $x_0\in V(x^*, r)$ and ${x_n}$ be the sequence generated by two-step Newton type  method given in $(\ref{eqn:12})$. Let $r$, $q_1$ and $q_2$ be determined by the expressions $(\ref{eqn:522})$ and $(\ref{eqn:524})$, respectively. Assume that $x_n\in V(x^*, r)$. Then
\begin{eqnarray}\label{eqn:526}
||y_n-x^*||&=&||y_n-x^*-[t'(x_n)]^{-1}t(x_n)||\nonumber\\
&=&||[t'(x_n)]^{-1}[t'(x_n)(x_n-x^*)-t(x_n)+t(x^*)]||.\nonumber\\
\end{eqnarray}
Expanding $t(x_n)$  along $\ x^*$ from Taylor's Expansion, we have
{\small
\begin{eqnarray}\label{eqn:527}
t(x^*)-t(x_n)+t'(x_n)(x_n-x^*)=t'(x^*)\int_{0}^{1}[t'(x^*)]^{-1}[t'(x_n)-t'(x)^\tau]d\tau(x_n-x^*).\nonumber\\
\end{eqnarray}}
Following  the hypothesis $(\ref{eqn:521})$ of the theorem and using the equations $(\ref{eqn:526})$ and $(\ref{eqn:527})$, it can be written as
{\small
\begin{eqnarray}\label{eqn:528}
||y_n-x^*||&\leq&||[t'(x_n)]^{-1}t'(x^*)||.||\int_{0}^{1}[t'(x^*)]^{-1}[t'(x_n)-t'(x^*)+t'(x^*)-t'(x)^\tau]d\tau||\nonumber\\
&&.||(x_n-x^*)||\nonumber\\
&\leq&\frac{1}{(1-\int_{0}^{2\rho(x_n)}L(u)du}\left(\int_{0}^{1}\int_{0}^{2\tau\rho(x_n)}L(u)du\rho(x_n)d\tau\right)\nonumber\\
&&+\left(\int_{0}^{1}\int_{0}^{2\rho(x_n)}L(u)du\rho(x_n)d\tau\right).
\end{eqnarray}}
In view of Lemma $(\ref{lm:21})$,  the above inequality becomes
\begin{eqnarray*}\label{eqn:529}
||y_n-x^*||&\leq&\frac{2\int_{0}^{2\rho(x_n)}L(u)du\rho(x_n)-\frac{1}{2}\int_{0}^{2\rho(x_n)}L(u)udu}{1-\int_{0}^{2\rho(x_n)}L(u)du}\nonumber\\
&\leq& \frac{2\int_{0}^{2\rho(x_n)}L(u)du}{1-\int_{0}^{2\rho(x_n)}L(u)du}\rho(x_n)= q_1\rho(x_n),
\end{eqnarray*}
which is same as first inequality of  $(\ref{eqn:520a})$. By similar analogy and form the final sub-step of the scheme $(\ref{eqn:12})$, we can write
{\small
\begin{eqnarray}\label{eqn:5210}
||x_{n+1}-x^*||&\le&||[t'(x_n)]^{-1}t'(x^*)||\left(||\int_{0}^{1}[t'(x^*)]^{-1}[t'(x_n)-t'(x^*)]d\tau||.||(y_n-x^*)||\right)\nonumber\\
&&+\left(||\int_{0}^{1}[t'(x^*)]^{-1}[t'(x^*)-t'(y^\tau)]d\tau||.||(y_n-x^*)||\right)\nonumber\\ 
&\leq&\frac{1}{(1-\int_{0}^{2\rho(x_n)}L(u)du}\left(\int_{0}^{1}\int_{0}^{2\tau\rho(y_n)}L(u)du\rho(y_n)d\tau\right)\nonumber\\
&&+\left(\int_{0}^{1}\int_{0}^{2\rho(x_n)}L(u)du\rho(y_n)d\tau\right).
\end{eqnarray}}
By virtue of Lemma $(\ref{lm:21})$, the above expression becomes \\
{\small
\begin{eqnarray*}\label{eqn:5211}
||x_{n+1}-x^*||&\leq&\frac{\int_{0}^{2\rho(x_n)}L(u)du\rho(y_n)+\int_{0}^{2\rho(y_n)}L(u)du\rho(y_n)-\frac{1}{2}\int_{0}^{2\rho(y_n)}L(u)udu}{1-\int_{0}^{2\rho(x_n)}L(u)du}\nonumber\\
&\leq& \frac{\int_{0}^{2\rho(x_n)}L(u)du\rho(y_n)+\int_{0}^{2\rho(y_n)}L(u)du\rho(y_n)}{1-\int_{0}^{2\rho(x_n)}L(u)du}\nonumber\\
&=& q_2q_1\rho(x_n),
\end{eqnarray*}}
where $q_1<1$ and $q_2<1$ are determined by the relation $(\ref{eqn:524})$.
Also, it can be seen that inequality $(\ref{eqn:523})$ may be easily derived from the second expression $(\ref{eqn:520a})$ and hence $x_n$ converges to $x^*$.\

 Furthermore, if the function $L_a$ defined by the relation $(\ref{eqn:56})$ is non-decreasing for some $a$ with $0\le a\le1$ and $r$ is determined by the inequality $(\ref{eqn:522})$, it follows from the first inequality of the expression $(\ref{eqn:520a})$ and Lemma $(\ref{lm:22})$ that
\begin{eqnarray*}
||y_n-x^*||&\leq&\frac{2 \varphi_{0,a}(2\rho(x_n))2^a}{(1-\int_{0}^{2\rho(x_n)} L(u) du)}\rho(x_n)^{a+1},\\
&\le&\frac{2\varphi_{0,a}(2\rho(x_0))2^a}{(1-\int_{0}^{2\rho(x_0)} L(u) du)}\rho(x_n)^{a+1}
=\frac{q_1}{\rho(x_0)^a}\rho(x_n)^{a+1}.
\end{eqnarray*}
Moreover, from the second inequality of expression $(\ref{eqn:520a})$ and Lemma $(\ref{lm:22})$, we get
 \begin{eqnarray*}
||x_{n+1}-x^*||&\leq&\frac{ \varphi_{0,a}(2\rho(x_n))+ \varphi_{0,a}(2\rho(y_n)).(2\rho(x_n))^a.\rho(y_n)}{(1-\int_{0}^{2\rho(x_n)} L(u) du)}\rho(y_n),\\
&\leq&\frac{ \varphi_{0,a}(2\rho(x_0))+ \varphi_{0,a}(2\rho(y_0)).(2\rho(x_n))^a.\rho(y_n)}{(1-\int_{0}^{2\rho(x_n)} L(u) du)}\rho(y_n),\\
&=&\frac{q_2.q_1}{\rho(x_0)^a\rho(y_0)^a}\rho(x_n)^{2a+1}.
\end{eqnarray*}
Hence, it can be seen that inequality $(\ref{eqn:525})$ may be easily derived and hence $x_n$ converges to $x^*$.
\end{proof}


Now, we wil apply our newly improved theorems to some special functions $L$ and results from Theorems $(\ref{th:51})$ and $(\ref{th:52})$ are recaptured.
								
\begin{Corollary}\label{cr:53}
Suppose that $x^*$ satisfies $t(x^*)=0$, $t$ has a continuous derivative in  $V(x^*, r)$, $[t'(x^*)]^{-1}$ exists and $[t'(x^*)]^{-1}t'$ satisfies the radius Lipschitz condition with the $L$-average with $L(u)=cau^{a-1}$i.e. :
\begin{equation}\label{eqn:531}
||[t'(x^*)]^{-1}(t'(x)-t'(y^\tau))||\leq c.(1-\tau^a)(||x-x^*||+||y-x^*||)^a, 
\end{equation}
$\forall \ x,\ y \in V(x^*, r),0\leq\tau\leq 1$, where $y^\tau=x^*+\tau(y-x^*)$, $\rho(x)=||x-x^*||$,$0<a<1$ and $c>0$. Let $r$ satisfies
\begin{eqnarray}\label{eqn:532}
r&=&\left(\frac{a+1}{c2^a(1+2a)}\right)^{\frac{1}{a}}.
\end{eqnarray}
Then two-step Newton type method $(\ref{eqn:12})$ is convergent for all  \ $x_0\in V(x^*, r)$ and 
\begin{eqnarray}\label{eqn:535}
||y_n-x^*||&\leq&\frac{ \int_{0}^{2\rho(x_n)} L(u)u du}{2(1-\int_{0}^{2\rho(x_n)} L(u) du)}\leq q_1\rho(x_n),
\end{eqnarray}
\begin{eqnarray}\label{eqn:536}
||x_{n+1}-x^*||&\leq&\frac{ \int_{0}^{\rho(x_n)+\rho(y_n)} L(u)u du}{(\rho(x_n)+\rho(y_n))(1-\int_{0}^{2\rho(x_n)} L(u) du)}\rho(y_n)\leq q_2 q_1\rho(x_n),\nonumber\\
\end{eqnarray}
where the quantities
\begin{eqnarray}\label{eqn:534}
q_1=\frac{ca2^a\rho(x_0)^a}{(1+a)[1-2^ac\rho(x_0)^a]} ,\ q_2=\frac{ca(\rho(x_0)+\rho(y_0))^a}{(a+1)(1-2^ac\rho(x_0)^a)},
\end{eqnarray}
are less than $1$. Furthermore,
\begin{eqnarray}\label{eqn:533}
||x_n-x^*||\leq C^{3^n-1}||x_0-x^*||,\ n=1,2,..., C=q_1\frac{\rho(x_0)}{\rho(y_0)}.
\end{eqnarray}
\end{Corollary}
\begin{Corollary}\label{cr:54}
Suppose that $x^*$ satisfies $t(x^*)=0$, $t$ has a continuous derivative in  $V(x^*, r)$, $[t'(x^*)]^{-1}$ exists and $[t'(x^*)]^{-1}t'$ satisfies the center Lipschitz condition with the $L$-average with $L(u)=cau^{a-1}$ i.e.:
\begin{eqnarray}\label{eqn:541}
||[t'(x^*)]^{-1}(t'(x)-t'(x^*))||\leq  c2^a||x-x^*||^a, \forall \ x \in V(x^*, r),
\end{eqnarray}
where $\rho(x)=||x-x^*||$, $0<a<1$ and $c>0$. Let $r$ satisfies
\begin{eqnarray}\label{eqn:542}
r&=&\left(\frac{1}{3c2^a}\right)^{\frac{1}{a}}.
\end{eqnarray}
Then two-step Newton type method $(\ref{eqn:12})$ is convergent for all  \ $x_0\in V(x^*, r)$ and 
\begin{eqnarray}\label{eqn:545}
||y_n-x^*||&\leq&\frac{2 \int_{0}^{2\rho(x_n)} L(u) du}{1-\int_{0}^{2\rho(x_n)} L(u) du}\rho(x_n)\leq q_1\rho(x_n), \nonumber\\
||x_{n+1}-x^*||&\leq&\frac{ \int_{0}^{2\rho(x_n)} L(u) du+ \int_{0}^{2\rho(y_n)} L(u) du}{1-\int_{0}^{2\rho(x_n)} L(u) du}\rho(y_n)\leq q_2 q_1\rho(x_n),\nonumber\\
\end{eqnarray}
where the quantities
\begin{eqnarray}\label{eqn:544}
q_1=\frac{c2^{a+1}\rho(x_0)^a}{[1-2^ac\rho(x_0)^a]} ,\ q_2=\frac{c2^a(\rho(x_0)^a+\rho(y_0)^a)}{(1-2^ac\rho(x_0)^a)},
\end{eqnarray}
are less than 1. Furthermore,
\begin{eqnarray}\label{eqn:543}
||x_n-x^*||\leq C^{3^n-1}||x_0-x^*||,\ n=1,2,..., C=q_1\frac{\rho(x_0)}{\rho(y_0)}
\end{eqnarray}
\end{Corollary}

\begin{Corollary}\label{cr:55}
Suppose that $x^*$ satisfies $t(x^*)=0$, $t$ has a continuous derivative in  $V(x^*, r)$, $[t'(x^*)]^{-1}$ exists and $[t'(x^*)]^{-1}t'$ satisfies the center Lipschitz condition with the $L$-average with $L(u)=\frac{2\gamma c}{(1-\gamma u)^3}$ i.e.:
\begin{eqnarray}\label{eqn:551}
||[t'(x^*)]^{-1}(t'(x)-t'(x^*))||\leq  \frac{c}{(1-2\gamma\rho(x))^2}-c, \forall \ x \in V(x^*, r)
\end{eqnarray}
where $\rho(x)=||x-x^*||$, $\gamma>0$ and $c>0$. Let $r$ satisfies
\begin{eqnarray}\label{eqn:552}
r&=&\frac{3c+1-\sqrt{3c(3c+1)}}{2\gamma(3c+1)}.
\end{eqnarray}
Then two-step Newton type method $(\ref{eqn:12})$ is convergent for all  \ $x_0\in V(x^*, r)$ and 
\begin{eqnarray}\label{eqn:555}
||y_n-x^*||&\leq&\frac{2 \int_{0}^{2\rho(x_n)} L(u) du}{1-\int_{0}^{2\rho(x_n)} L(u) du}\rho(x_n)\leq q_1\rho(x_n), \nonumber\\
||x_{n+1}-x^*||&\leq&\frac{ \int_{0}^{2\rho(x_n)} L(u) du+ \int_{0}^{2\rho(y_n)} L(u) du}{1-\int_{0}^{2\rho(x_n)} L(u) du}\rho(y_n)\leq q_2 q_1\rho(x_n),\nonumber\\
\end{eqnarray}
where the quantities
{\small
\begin{eqnarray}\label{eqn:554}
q_1&=&\frac{2c-2c(1-2\gamma\rho(x_0))^2}{[1-2\gamma \rho(x_0)]^2(1+c)-c},\\
q_2&=&\frac{[c-c(1-2\gamma\rho(x_0))^2](1-2\gamma\rho(y_0))^2)+[c-c(1-2\gamma\rho(y_0))^2](1-2\gamma\rho(x_0))^2)}{([1-2\gamma \rho(x_0)]^2(1+c)-c)(1-2\gamma\rho(y_0))^2)},\nonumber\\
\end{eqnarray}}
are less than 1. Furthermore,
\begin{eqnarray}\label{eqn:553}
||x_n-x^*||\leq C^{3^n-1}||x_0-x^*||,\ n=1,2,..., C=q_1\frac{\rho(x_0)}{\rho(y_0)}.  
\end{eqnarray}
\end{Corollary}
\textbf{Example 5.1}
Let $X=Y=R$, the reals. Define
\begin{eqnarray*}\label{eqn:ex1}
t(x)= \int_{0}^{x}\left(1+2x\sin\frac{\pi}{x}\right)dx,\ \forall x \in R.  
\end{eqnarray*}
Then
\[
  t'(x) =
  \begin{cases}
                                   1+2x\sin\frac{\pi}{x}, & x\neq0, \\
                                   1, & x=0, 
  \end{cases}
\]
Obviously, $x^*=0$ is a zero of $t$ and $t'$ satisfies that 
\begin{eqnarray*}\label{eqn:ex2}
||[t'(x^*)]^{-1}(t'(x)-t'(x^*))||= \left|2x\sin\frac{\pi}{x}\right|\leq2|x-x^*|, \forall \ x \in R.
\end{eqnarray*}
It follows from Theorem $(\ref{th:52})$ that for any $x_0\in V(x^*,1/6)$
\begin{eqnarray*}\label{eqn:ex3}
||x_n-x^*||\leq C^{3^n-1}||x_0-x^*||,\ n=1,2,\cdots, C=\left(\frac{4|x_0|^2}{1-2|x_0|.|y_0|}\right).
\end{eqnarray*}
However, there is no positive integrable function $L$ such that the\\ inequality $(\ref{eqn:51})$ is satisfied. In fact, notice that 
 \begin{equation*}\label{eqn:ex4}
||[t'(x^*)]^{-1}(t'(x)-t'(y^\tau))||=\left|2x\sin\frac{\pi}{x}-2y\tau\sin\frac{\pi}{y\tau}\right|=\frac{4}{2k+1},
\end{equation*}
for $x=1/k, y=1/k, \tau=\frac{2k}{2k+1}$ and $k=1,2,\cdots$ Thus, if there was a positive integrable function $L$ such that the inequality $(\ref{eqn:51})$ holds on $V(x^*, r)$ for some $r>0$, it follows that there exists some $n_0>1$ such that
\begin{eqnarray*}\label{eqn:ex5}
\int_{0}^{2r}L(u)du\geq \sum_{k=n_0}^{+\infty}\int_{\frac{4}{2k+1}}^{\frac{2}{k}}L(u)du\geq \sum_{k=n_0}^{+\infty}\frac{4}{2k+1}=+\infty,
\end{eqnarray*}
which is a contradiction. This example shows that Theorem $(\ref{th:52})$ is a crucial improvement of Theorem $(\ref{th:51})$ if the radius of the convergence ball is ignored.







\begin{thebibliography}{10}
\bibitem{Cho}
Argyros, I. K., Cho, Y. J., George, S. (2016): Local convergence for some third-order iterative methods under weak conditions, J. Korean Math. Soc, 53 (4), 781-793.
\bibitem{Gonzalez}
Argyros, I. K., Gonzalez, D. (2015): Local convergence for an improved Jarratt-type method in Banach space, International Journal of Interactive Multimedia and Artificial Intelligence, 3 (4), 20-25.
\bibitem{Chen}
Chen, J., Li, W. (2006): Convergence behavior of inexact Newton methods under weak Lipschitz condition, Journal of Computational and Applied Mathematics, 191 (1), 143-164.
\bibitem{Homeier}
Homeier, H. H. H. (2005): On Newton-type methods with cubic convergence. Journal of computational and Applied Mathematics, 176(2), 425-432.
\bibitem{Kantorovich}
 Kantorovich, L. V., Akilov, G. P. (1982): Functional Analysis, Pergamon Press, Oxford.
\bibitem{Kanwar}
Kanwar, V., Kukreja, V. K., Singh, S. (2005): On some third-order iterative methods for solving nonlinear equations,  Applied Mathematics and Computation, 171 (1), 272-280.
\bibitem{Li}
Kou, J., Li, Y., Wang, X. (2006):  A modification of Newton method with third-order convergence, Applied Mathematics and Computation, 181 (2), 1106-1111.
\bibitem{Kang}
Nazeer, W., Tanveer, M., Kang, S. M., Naseem, A. (2016): A new Householder’s method free from second derivatives for solving nonlinear equations and polynomiography,  J. Nonlinear Sci. Appl, 9 (3), 998-1007.
\bibitem{Ortega}
Ortega, J. M., Rheinboldt, W. C. (2000):  Iterative solution of nonlinear equations in several variables. Society for Industrial and Applied Mathematics, Philadelphia.
\bibitem{Rall}
Rall, L. B. (1969): Computational Solution of Nonlinear Operator Equations, John Wiley and Sons, Inc, USA.
\bibitem{Argyros}
Ruiz, A. M., Argyros, I. K. (2014): Two-step Newton methods, Journal of Complexity, 30 (4), 533-553.
\bibitem{Shakhno}
Shakhno, S. (2010): On a two-step iterative process under generalized Lipschitz conditions for first-order divided difference, Journal of Mathematical Sciences, 168 (4), 576-584.
\bibitem{Sharma}
Sharma, D., Parhi S. K. (2020): On the local convergence of modified Weerakoon’s method in Banach spaces, The Journal of Analysis, 28 (3), 867-877.
\bibitem{Traub}
Traub, J. F. (1977): Iterative Methods for the Solution of Equations, Chelsea Publishing Company, New York.
\bibitem{Wang}
Wang, X. (2000): Convergence of Newton's method and uniqueness of the solution of equations in Banach space, IMA Journal of Numerical Analysis, 20 (1), 123-134.
\bibitem{Wang2}
Wang, X. H., Li, C. (2003): Convergence of Newton's method and uniqueness of the solution of equations in Banach spaces II, Acta Mathematica Sinica, 19 (2), 405-412.
\end{thebibliography}
\end{document}